\documentclass[12pt]{amsart}
\usepackage{amssymb,latexsym}
\usepackage{amsmath}
\usepackage{amsfonts}
\usepackage[colorlinks,linkcolor=blue,anchorcolor=blue,citecolor=blue]{hyperref}

\newtheorem{theorem}{Theorem}[section]
\newtheorem{lemma}[theorem]{Lemma}
\newtheorem{definition}[theorem]{Definition}
\newtheorem{example}[theorem]{Example}

\newtheorem{proposition}[theorem]{Proposition}

\newtheorem{corollary}[theorem]{Corollary}

\newcommand\ord{{\rm ord}}
\newcommand\lcm{{\rm lcm}}

\newcommand\Z{{\mathbb Z}}

\numberwithin{equation}{section}

\begin{document}

\title{The Arithmetic of Carmichael Quotients}

\author{Min Sha}
\address{School of Mathematics and Statistics, University of New South Wales,
 Sydney, NSW 2052, Australia}
\email{shamin2010@gmail.com}






\begin{abstract}
Carmichael quotients for an integer $m\ge 2$ are introduced analogous to Fermat quotients, by using Carmichael function $\lambda(m)$.
Various properties of these new quotients are investigated, such as basic arithmetic properties, sequences derived from Carmichael quotients, Carmichael-Wieferich numbers, and so on. Finally, we link Carmichael quotients to perfect nonlinear functions.

\end{abstract}

\maketitle



\section{Introduction}
Let $p$ be a prime and $a$ an integer not divisible by $p$, by Fermat's little theorem, the \textit{Fermat quotient} of $p$ with base $a$ is defined as follows
$$
Q_{p}(a)=\frac{a^{p-1}-1}{p}.
$$
Moreover, if $Q_{p}(a)\equiv 0$ (mod $p$), then we call $p$ a \textit{Wieferich prime} with base $a$.

This quotient has been extensively studied from various aspects because of its numerous applications in number theory
and computer science; see, for example, \cite{Chen,CD,COW,Gran,OS,R}.
A first comprehensive study of Fermat quotient was published in 1905 by Lerch \cite{Lerch}, which was based on the viewpoint of arithmetic. More arithmetic properties were investigated in \cite{TA3}.

In \cite{ADS}, the authors generalized the definition of Fermat quotient by using Euler's theorem. Let $m\ge 2$ and $a$ be relatively prime integers, the \textit{Euler quotient} of $m$ with base $a$ is defined as follows
$$
Q_{m}(a)=\frac{a^{\varphi(m)}-1}{m},
$$
where $\varphi$ is Euler's totient function. Moreover, if $Q_{m}(a)\equiv 0$ (mod $m$), then we call $m$ a \textit{Wieferich number} with base $a$. They also undertook a very careful study of Euler quotients.

In fact, there are some other generalizations of Fermat quotients, see \cite{TA1,SS,Skula1}. Especially, in
\cite{TA1} the author introduced a quotient like $(a^{e}-1)/m$, where $\gcd(a,m)=1$ and
$e$ is the multiplicative order of $a$ modulo $m$.

In this paper, we introduce a different generalization of Fermat quotient by using Carmichael function and study its arithmetic properties.

For a positive integer $m$, the Carmichael function $\lambda(m)$ is defined to be the exponent of the multiplicative group $(\Z/m\Z)^{*}$. More explicitly, $\lambda(1)=1$; for a prime power $p^{r}$
we define
\begin{equation}
\lambda(p^{r})=\left\{ \begin{array}{ll}
                 p^{r-1}(p-1) & \textrm{if $p\ge 3$ or $r\le 2$},\\
                 2^{r-2} & \textrm{if $p=2$ and $r\ge 3$};
                 \end{array} \right.
\notag
\end{equation}
and
\begin{equation}
\lambda(m)={\rm lcm}(\lambda(p_{1}^{r_{1}}),\lambda(p_{2}^{r_{2}}),\cdots,\lambda(p_{k}^{r_{k}})),
\notag
\end{equation}
where, as usual, ``lcm" means the least common multiple, and  $m=p_{1}^{r_{1}}p_{2}^{r_{2}}\cdots p_{k}^{r_{k}}$ is the prime factorization of $m$.

For every positive integer $m$, we have $\lambda(m)|\varphi(m)$, and
$\lambda(m)=\varphi(m)$ if and only if $m\in \{1,2,4,p^{k},2p^{k}\}$, where $p$ is an odd prime and $k\ge 1$.
In addition, if $m|n$, we have $\lambda(m)|\lambda(n)$.

\begin{definition}
{\rm Let $m\ge 2$ and $a$ be relatively prime integers. The quotient
\begin{equation}
C_{m}(a)=\frac{a^{\lambda(m)}-1}{m}
\notag
\end{equation}
is called the \textit{Carmichael quotient} of $m$ with base $a$. Moreover, if $C_{m}(a)\equiv 0$ $({\rm mod}\, m)$, we call  $m$ a \textit{Carmichael-Wieferich number} with base $a$.
}
\end{definition}

We want to indicate that the term ``Carmichael quotient" was introduced in \cite{TA2} to denote a different quotient, and we think that there is no much danger of confusion.

We extend many known results about Fermat quotients or Euler quotients to Carmichael quotients by using the same techniques,
such as basic arithmetic properties with special emphasis on congruences, the least periods of
 sequences derived from Carmichael quotient, Carmichael-Wieferich numbers. Finally, we link Carmichael quotients to perfect nonlinear functions.

\section{Arithmetic of Carmichael Quotients}
In what follows, we fix $m\ge 2$ an integer unless stated otherwise.

In this section, we study some basic arithmetic properties of Carmichael quotients and extend some results about Fermat quotients or Euler quotients in \cite{ADS,Lerch,Lerch2}. See \cite{ADS} for historical literatures.

For any integer $a$ with $\gcd(a,m)=1$, we have $C_{m}(a)|Q_{m}(a)$. In particular, $C_m(a)=Q_m(a)$ when $m$ is an odd prime power. Furthermore, it is straightforward to prove that
they have the following relation.
\begin{proposition}\label{relation}
For any integer $a$ with $\gcd(a,m)=1$, we have
$$
Q_{m}(a)\equiv \frac{\varphi(m)}{\lambda(m)}\cdot C_{m}(a)\quad ({\rm mod}\, m).
$$
\end{proposition}
\begin{proof}
Since $\lambda(m)|\varphi(m)$, we derive 
\begin{align*}
Q_m(a)&=\frac{(a^{\lambda(m)})^{\varphi(m)/\lambda(m)}-1}{m}\\
&=\frac{(a^{\lambda(m)}-1)\left(1+a^{\lambda(m)}+
\cdots+(a^{\lambda(m)})^{\varphi(m)/\lambda(m)-1}\right)}{m}\\
& \equiv \frac{\varphi(m)}{\lambda(m)} C_{m}(a)\quad ({\rm mod}\, m).
\end{align*}
 
\end{proof}

Now we state two fundamental congruences for Carmichael quotients, which are crucial for further study.
\begin{proposition}\label{fun}
{\rm (1)} If a and b are integers with $\gcd(ab,m)=1$, then we have
$$
C_{m}(ab)\equiv C_{m}(a)+C_{m}(b)\quad ({\rm mod}\, m).
$$
{\rm (2)} If a, k are integers with $\gcd(a,m)=1$, and $\alpha$ is a positive integer, then we have
$$
C_{m}(a+km^{\alpha})\equiv C_{m}(a)+\frac{k\lambda(m)}{a}m^{\alpha-1}\quad ({\rm mod}\, m^{\alpha}).
$$
\end{proposition}
\begin{proof}
(1) We only need to notice that 
\begin{align*}
C_m(ab)&=\frac{a^{\lambda(m)}b^{\lambda(m)}-1}{m} \\
&=\frac{(a^{\lambda(m)}-1)(b^{\lambda(m)}-1)+(a^{\lambda(m)}-1)+(b^{\lambda(m)}-1)}{m}.
\end{align*}

(2) Using the binomial expansion, it is easy to see that  
$$
C_{m}(a+km^{\alpha})\equiv \frac{a^{\lambda(m)}+\lambda(m)a^{\lambda(m)-1}km^{\alpha}-1}{m}\quad ({\rm mod}\, m^{\alpha}),
$$
which implies the desired congruence.
 
\end{proof}

The following two corollaries concern some short sums of Carmichael quotients.
\begin{corollary}\label{sum1}
If $m\ge 3$, for any integer $a$ with $\gcd(a,m)=1$, we have
$$
\sum\limits_{k=0}^{m-1}C_{m}(a+km)\equiv 0\quad ({\rm mod}\, m).
$$
\end{corollary}
\begin{proof}
First applying Proposition \ref{fun} (2) and then noticing that $\lambda(m)$ is even when $m\ge 3$, we obtain
$$
\sum\limits_{k=0}^{m-1}C_{m}(a+km)\equiv \frac{\lambda(m)}{a}\cdot\frac{m(m-1)}{2} \equiv 0 \quad ({\rm mod}\, m).
$$
 
\end{proof}

\begin{corollary}\label{sum2}
If $m\ge 3$, for any integer $a$ with $\gcd(a,m)=1$, we have
\begin{equation}
\sum\limits_{\substack{a=1\\ \gcd(a,m)=1}}^{m^{2}}C_{m}(a)\equiv 0\quad ({\rm mod}\, m).
\notag
\end{equation}
\end{corollary}
\begin{proof}
Notice that
$$
\sum\limits_{\substack{a=1\\ \gcd(a,m)=1}}^{m^{2}}C_{m}(a)=
\sum\limits_{\substack{a=1\\ \gcd(a,m)=1}}^{m}\sum\limits_{k=0}^{m-1}C_{m}(a+km).
$$
Then, the desired result follows from Corollary \ref{sum1}.  
\end{proof}

We want to remark that the results in Corollaries \ref{sum1} and \ref{sum2} are not true when $m=2$.

The next proposition concerns some relationships between various $C_{m}(a)$ with fixed base $a$ and different moduli.
\begin{proposition}
{\rm (1)} If $\gcd(a,mn)=1$, then
$$
C_{m}(a)|nC_{mn}(a).
$$
{\rm (2)} If $\gcd(a,mn)=\gcd(m,n)=1$, then
$$
C_{mn}(a)\equiv \frac{\lambda(n)}{n\cdot \gcd(\lambda(m),\lambda(n))}C_{m}(a)\quad ({\rm mod}\, m).
$$
{\rm (3)} Assume that $\gcd(a,mn)=\gcd(m,n)=1$, and let X and Y be two integers satisfying $m^{2}X+n^{2}Y=1$. Then
$$
C_{mn}(a)\equiv \frac{n\lambda(n)}{\gcd(\lambda(m),\lambda(n))}YC_{m}(a)+\frac{m\lambda(m)}{\gcd(\lambda(m),\lambda(n))}XC_{n}(a)\quad ({\rm mod}\, mn).
$$
\end{proposition}
\begin{proof}
(2) Under the assumption, noticing that $\lambda(mn)=\frac{\lambda(m)\lambda(n)}{\gcd(\lambda(m),\lambda(n))}$, we have
\begin{equation}\label{p1}
\begin{array}{lll}
C_{mn}(a)=\frac{a^{\frac{\lambda(m)\lambda(n)}{\gcd(\lambda(m),\lambda(n))}}-1}{mn}
&=&\frac{\left(a^{\lambda(m)}\right)^{\frac{\lambda(n)}{\gcd(\lambda(m),\lambda(n))}}-1}{mn}\\
&\equiv & \frac{\lambda(n)(a^{\lambda(m)}-1)}{mn\cdot \gcd(\lambda(m),\lambda(n))} \quad ({\rm mod}\, m).
\end{array}
\notag
\end{equation}

(3) It suffices to show that the equality is true for modulo $m$ and modulo $n$ respectively.
But this follows directly from (2).  
\end{proof}

For any integer $a$ with $\gcd(a,m)=1$, we denote $\langle a\rangle$ as the subgroup of $(\mathbb{Z}/m\mathbb{Z})^{*}$ generated by $a$,
and we let $\ord_{m}a$ be the multiplicative order of $a$ modulo $m$. The following expression is so-called Lerch's expression \cite{Lerch2}.
\begin{proposition}\label{con1}
If $\gcd(a,m)=1$ and assume $n=\ord_{m}a$, then
\begin{equation}
C_{m}(a)\equiv\frac{\lambda(m)}{n}\sum\limits_{\substack{r=1\\ r\in\langle a\rangle}}^{m}\frac{1}{ar}\left\lfloor\frac{ar}{m}\right\rfloor
 \quad ({\rm mod}\, m),
\notag
\end{equation}
where $\lfloor x\rfloor$ denotes the greatest integer $\le x$.
\end{proposition}
\begin{proof}
For each $1\le r \le m$ with $r\in \langle a\rangle$, we write $ar\equiv c_{r} ({\rm mod}\, m)$, with $1\le c_{r}\le m$.
Notice that when $r$ runs through all elements with $1\le r \le m$ and $r\in \langle a\rangle$, so does $c_{r}$. Let $P$ denote the product of all such integers $c_{r}$.
If the products and sums below are understood to be taken over all $r$ with $1\le r\le m$ and $r\in \langle a\rangle$, we have
$$
P^{\frac{\lambda(m)}{n}}=\prod c_{r}^{\frac{\lambda(m)}{n}}=\prod\left(ar-m\left\lfloor\frac{ar}{m} \right\rfloor\right)^{\frac{\lambda(m)}{n}}
=a^{\lambda(m)}P^{\frac{\lambda(m)}{n}}\prod\left(1-\frac{m}{ar}\left\lfloor\frac{ar}{m} \right\rfloor\right)^{\frac{\lambda(m)}{n}}.
$$
So
$$
1=a^{\lambda(m)}\prod\left(1-\frac{m}{ar}\left\lfloor\frac{ar}{m} \right\rfloor\right)^{\frac{\lambda(m)}{n}}
\equiv a^{\lambda(m)}\left(1-m\sum\frac{1}{ar}\left\lfloor\frac{ar}{m} \right\rfloor\right)^{\frac{\lambda(m)}{n}} \quad ({\rm mod}\, m^{2}).
$$
Then we get
$$
a^{\lambda(m)}-1\equiv a^{\lambda(m)}\frac{m\lambda(m)}{n}\sum\limits_{\substack{r=1\\ r\in\langle a\rangle}}^{m}\frac{1}{ar}\left\lfloor\frac{ar}{m}\right\rfloor
 \quad ({\rm mod}\, m^{2}),
$$
which implies the desired congruence.  
\end{proof}

In the last part of this section, we describe the decomposition of Carmichael quotients in the dependence of the prime factorization of the modulus. Further we investigate Carmichael quotients for prime power moduli.

\begin{proposition}\label{factor}
Let $m=p_{1}^{r_{1}}\cdots p_{k}^{r_{k}}$ be the prime factorization of $m$,
and let $a$ be an integer with $\gcd(a,m)=1$. For $1\le i\le k$, let $d_{i}=\lambda(m)/\lambda(p_{i}^{r_{i}})$,
$m_{i}=m/p_{i}^{r_{i}}$ and $m^{\prime}_{i}\in \mathbb{Z}$ such that $m_{i}^{2}m^{\prime}_{i}\equiv 1$ {\rm (mod $p_{i}^{r_{i}}$)}. Then
$$
C_{m}(a)\equiv \sum\limits_{i=1}^{k}m_{i}m^{\prime}_{i}d_{i}C_{p_{i}^{r_{i}}}(a)\quad ({\rm mod}\, m).
$$
\end{proposition}
\begin{proof}
It suffices to prove for each $1\le j\le k$,
$$
C_{m}(a)\equiv \sum\limits_{i=1}^{k}m_{i}m^{\prime}_{i}d_{i}C_{p_{i}^{r_{i}}}(a)\quad ({\rm mod}\, p_{j}^{r_{j}}),
$$
that is
$$
C_{m}(a)\equiv m_{j}m^{\prime}_{j}d_{j}C_{p_{j}^{r_{j}}}(a)\quad ({\rm mod}\,p_{j}^{r_{j}}).
$$
Since we have
\begin{equation}
C_{m}(a)=\frac{a^{\lambda(p_{j}^{r_{j}})d_{j}}-1}{m}\equiv \frac{d_{j}(a^{\lambda(p_{j}^{r_{j}})}-1)}{m}\equiv
m_{j}m^{\prime}_{j}d_{j}C_{p_{j}^{r_{j}}}(a)\quad ({\rm mod}\, p_{j}^{r_{j}}),
\notag
\end{equation}
the result follows.  
\end{proof}

\begin{proposition}\label{reduction}
Let $p$ be an odd prime and $\gcd(a,p)=1$. For any two integers $i$ and $j$ with $1\le i\le j$, we have
$$
C_{p^{j}}(a)\equiv C_{p^{i}}(a) \quad ({\rm mod}\, p^{i}).
$$
Besides, for $3\le i\le j$ and $\gcd(a,2)=1$, we have
$$
C_{2^{j}}(a)\equiv C_{2^{i}}(a) \quad ({\rm mod}\, 2^{i-1}).
$$
\end{proposition}
\begin{proof}
Notice that $C_{p^{i}}(a)=Q_{p^{i}}(a)$ if $p$ is an odd prime. By \cite[Proposition 4.1]{ADS}, for any integer $k\ge 1$, we have
$$
C_{p^{k+1}}(a)\equiv C_{p^{k}}(a) \quad ({\rm mod}\, p^{k}).
$$
Then the first formula follows.

Since for $r\ge 3$, we have
\begin{equation}
\begin{array}{lll}
C_{2^{r+1}}(a)-C_{2^{r}}(a)&\equiv& \frac{a^{2^{r-2}}-1}{2}C_{2^{r}}(a) \quad ({\rm mod}\, 2^{r})\\
&\equiv& 0 \quad ({\rm mod}\, 2^{r-1}),
\end{array}
\notag
\end{equation}
we get the second formula.   
\end{proof}

The following corollary, about the relation between Carmichael quotients and Fermat quotients,
can be obtained directly from the above two propositions.

\begin{corollary}\label{modp}
Suppose that $p$ is an odd prime factor of $m$, and $p^\alpha$ is the largest power of $p$ dividing $m$.
Let  $d_{1}=\frac{\lambda(m)}{\lambda(p^{\alpha})}$,
$m_{1}=m/p^{\alpha}$, and $m_{1}^{\prime}\in \mathbb{Z}$ such that $m_{1}^{2}m_{1}^{\prime}\equiv 1$ {\rm (mod $p^{\alpha}$)}. Then
for any integer $a$ with $\gcd(a,m)=1$, we have
$$
C_{m}(a)\equiv m_{1}m_{1}^{\prime}d_{1}Q_{p}(a)\quad ({\rm mod}\, p).
$$
\end{corollary}

\section{Sequences derived from Carmichael quotients}

In this section, we will define two periodic sequences by Carmichael quotients and determine their least (positive) periods
following the method in the proof of \cite[Proposition 2.1]{CW}.

As usual, for a periodic sequence $\{s_n\}_{n=1}^{\infty}$, a positive integer $j$ is called its \textit{period} if $s_{n+j}=s_n$ for any $n\ge 1$; if further $j$ is the smallest positive integer endowed with such property, we call $j$ the \textit{least period} of $\{s_n\}$.

Let $m=p_{1}^{r_{1}}\cdots p_{k}^{r_{k}}$ be the prime factorization of the integer $m$ ($m\ge 2$).
For each $1\le i\le k$, put $m_{i}=m/p_{i}^{r_{i}}$, and let $w_{i}$ be the integer defined by
$p_{i}^{w_{i}}=\gcd(\lambda(m)/\lambda(p_{i}^{r_{i}}),p_{i}^{r_{i}})$, here note that $0\le w_i \le r_i$. 

Now, we want to define a sequence $\{a_{n}\}$ following the manner in \cite{CW}.

First, for any integer $n$ and any $1\le i \le k$, 
if $p_{i}|n$, set $C_{p_{i}^{r_{i}}}(n)=0$.
Then, for every integer $n\ge 1$, by Proposition \ref{factor}, $a_{n}$ is defined as the unique integer with
$$
a_{n}\equiv \sum\limits_{i=1}^{k}\frac{m_{i}m^{\prime}_{i}\lambda(m)}{\lambda(p_{i}^{r_{i}})}C_{p_{i}^{r_{i}}}(n)\quad ({\rm mod}\, m), \qquad 0\le a_{n}\le m-1,
$$
where $m^{\prime}_{i}\in \mathbb{Z}$ is such that $m_{i}^{2}m^{\prime}_{i}\equiv 1$ {\rm (mod $p_{i}^{r_{i}}$)} for
each $1\le i\le k$. So, if $\gcd(n,m)=1$, we have $a_{n}\equiv C_{m}(n)$ (mod $m$).

By Proposition \ref{fun} (2), $m^{2}$ is a period of $\{a_{n}\}$. We denote its least period by $T$.
For each $1\le i\le k$, let $T_{i}$ be the least period of the sequence $\{a_{n}\mod p_{i}^{r_{i}}\}$. Obviously, we have
$$
T=\lcm(T_{1},\cdots,T_{k}).
$$
Thus, in order to determine $T$, it suffices to compute $T_{i}$ for each $1\le i\le k$.

For every $1\le i\le k$, we have
\begin{equation}\label{1period}
a_{n}\equiv \frac{\lambda(m)}{m_{i}\lambda(p_{i}^{r_{i}})}C_{p_{i}^{r_{i}}}(n)\quad ({\rm mod}\, p_{i}^{r_{i}}).
\end{equation}
So, $T_{i}$ equals to the least period of $\{C_{p_{i}^{r_{i}}}(n)\mod p_{i}^{r_{i}-w_{i}}\}$.
Here, we also denote $T_{i}$ as the least period of the sequence $\{C_{p_{i}^{r_{i}}}(n)\mod p_{i}^{r_{i}-w_{i}}\}$ without confusion.
In the sequel, we will calculate $T_{i}$ case by case for any fixed $1\le i\le k$.
\begin{lemma}
If $w_{i}=r_{i}$, then $T_{i}=1$.
\end{lemma}
\begin{proof}
Since in this case we have $C_{p_{i}^{r_{i}}}(n)\equiv 0$ (mod $p_{i}^{r_{i}-w_{i}}$) for all $n\ge 1$.  
\end{proof}

\begin{lemma}
If $p_{i}>2$ and $w_{i}<r_{i}$, then $T_{i}=p_{i}^{r_{i}-w_{i}+1}$.
\end{lemma}
\begin{proof}
Combining Proposition \ref{fun} (2) with Proposition \ref{reduction}, for integers $n$ and $\ell$ with $\gcd(n,p_{i})=1$, we have
\begin{align*}
C_{p_{i}^{r_{i}}}(n+\ell p_{i}^{r_{i}-w_{i}})
& \equiv C_{p_{i}^{r_{i}-w_i}}(n+\ell p_{i}^{r_{i}-w_{i}}) \\
& \equiv C_{p_{i}^{r_{i}-w_i}}(n)+\ell n^{-1}(p_i-1)p_{i}^{r_{i}-w_{i}-1}\\
& \equiv C_{p_{i}^{r_{i}}}(n)+\ell n^{-1}(p_i-1)p_{i}^{r_{i}-w_{i}-1}
\quad ({\rm mod}\, p_{i}^{r_{i}-w_{i}}).
\end{align*}
Thus, $T_{i}=p_{i}^{r_{i}-w_{i}+1}$.  
\end{proof}

Now, it remains to consider the case $p_i=2$.

\begin{lemma}
If $p_{i}=2$ and $w_{i}=0$, then
\begin{equation}
T_{i}=\left\{ \begin{array}{ll}
                 4 & r_{i}=1,\\

                 8 & r_{i}=2,\\

                 2^{r_{i}+2} & r_{i}\ge 3
                 \end{array} \right.
\notag
\end{equation}
\end{lemma}
\begin{proof}
Notice that for each $n$ with $\gcd(n,2)=1$, by Proposition \ref{fun} (2) we have
$$
C_{2^{r_{i}}}(n+\ell\cdot 2^{r_{i}})\equiv C_{2^{r_{i}}}(n)+\ell n^{-1}\lambda(2^{r_{i}})
\quad ({\rm mod}\, 2^{r_{i}}).
$$
Then, the result follows easily.  
\end{proof}

\begin{lemma}\label{2period}
For $r\ge 3$, the least period of the sequence $\{C_{2^{r+1}}(n)\mod 2^{r}\}$ is $2^{r+2}$.
\end{lemma}
\begin{proof}
For $r\ge 3$ and $\gcd(n,2)=1$, we have $C_{2^{r+1}}(n)=\frac{n^{2^{r-2}}+1}{2}C_{2^{r}}(n)$. Then using Proposition \ref{fun} (2), we deduce that 
\begin{equation}
\begin{array}{lll}
C_{2^{r+1}}(n+\ell\cdot 2^{r})-C_{2^{r+1}}(n)&=&\frac{n^{2^{r-2}}+1}{2}\left(C_{2^{r}}(n+\ell\cdot 2^{r})-C_{2^{r}}(n)\right)\\
&\equiv&\frac{n^{2^{r-2}}+1}{2}\cdot \ell n^{-1}2^{r-2} \quad ({\rm mod}\, 2^{r}),  
\end{array}
\notag
\end{equation}
which implies the desired result by noticing that $n^{2^{r-2}} \equiv 1 \pmod{2^r}$ and then  $\frac{n^{2^{r-2}}+1}{2}$ is odd.  
\end{proof}

\begin{lemma}
If $p_{i}=2$ and $3\le r_{i}-w_{i}<r_{i}$, then $T_{i}=2^{r_{i}-w_{i}+2}$.
\end{lemma}
\begin{proof}
By Proposition \ref{reduction}, for $\gcd(n,2)=1$, we have
$$
C_{2^{r_{i}}}(n)\equiv C_{2^{r_{i}-w_{i}+1}}(n)\quad ({\rm mod}\, 2^{r_{i}-w_{i}}).
$$
Then, the result follows directly from Lemma \ref{2period}.  
\end{proof}

\begin{lemma}
If $p_{i}=2$, $r_{i}\ge 3$ and $1\le r_{i}-w_{i}\le 2$, then $T_{i}=2^{r_{i}-w_{i}+2}$.
\end{lemma}
\begin{proof}
From Proposition \ref{reduction}, for $\gcd(n,2)=1$, we have
$$
C_{2^{r_{i}}}(n)\equiv C_{2^{3}}(n)\quad ({\rm mod}\, 2^{2}).
$$
So, $T_{i}$ equals to the least period of the sequence $\{C_{2^{3}}(n)\mod 2^{r_{i}-w_{i}}\}$.
 By Proposition \ref{fun} (2), we have
 $$
C_{2^{3}}(n+ \ell\cdot 2^{3})\equiv C_{2^{3}}(n)+2\ell n^{-1}
\quad ({\rm mod}\, 2^{2}),
$$
which implies the desired result. In fact, one can also verify this lemma by direct calculations.  
\end{proof}

\begin{lemma}
If $p_{i}=2$, $r_{i}= 2$ and $w_{i}=1$, then $T_{i}=1$.
\end{lemma}

We summarize the above results in the following proposition.
\begin{proposition}
For each $1\le i\le k$, if $p_{i}$ is an odd prime, then
\begin{equation}
T_{i}=\left\{ \begin{array}{ll}
                 1 & w_{i}=r_{i},\\

                 p_{i}^{r_{i}-w_{i}+1} & w_{i}<r_{i};
                 \end{array} \right.
\notag
\end{equation}
otherwise if $p_{i}=2$, then
\begin{equation}
T_{i}=\left\{ \begin{array}{ll}
                 1 & w_{i}=r_{i},\\

                 4 & r_{i}=1, w_{i}=0,\\

                 8 & r_{i}=2, w_{i}=0,\\

                 1 & r_{i}=2, w_{i}=1,\\

                 2^{r_{i}-w_{i}+2} & r_{i}\ge3, w_{i}<r_{i}.
                 \end{array} \right.
\notag
\end{equation}
In particular, the least period of $\{a_n\}$ is $T=T_{1}T_{2}\cdots T_{k}$.
\end{proposition}

When $m=p^r$ with $p$ an odd prime and $r\ge 1$, we have $T=p^{r+1}$, which is consistent with \cite[Proposition 2.1]{CW}. If $m=2^r$ with $r\ge 3$, then
$T=2^{r+2}$; but by \cite[Proposition 2.1]{CW}, the least period of the sequence defined there by Euler quotient is $2^{r+1}$.

Finally, we want to define a new sequence $\{b_{n}\}$, which is much simpler but has the same least period as $\{a_{n}\}$.

For an integer $n\ge 1$ with $\gcd(n,m)=1$, $b_{n}$ is defined to be the unique integer with
$$
b_{n}\equiv C_{m}(n)\quad ({\rm mod}\, m),\qquad 0\le b_{n}\le m-1;
$$
and we also define
$$
b_{n}=0, \qquad \textrm{if $\gcd(n,m)\neq1$}.
$$
Since $b_{n}$ also satisfies (\ref{1period}) for any integer $n$ with  $\gcd(n,m)=1$, the least period of $\{b_{n}\}$ equals to that
of  $\{a_{n}\}$.

\begin{proposition}
The sequence $\{b_n\}$ has the same least period as $\{a_n\}$.
\end{proposition}

\section{Carmichael-Wieferich Numbers}\label{number}
In this section, except for extending some results in \cite{ADS}, we study Carmichael-Wieferich numbers from more aspects,
especially Proposition \ref{limit}.

First, we want to deduce some basic facts for Carmichael-Wieferich numbers.

\begin{proposition}
If $m\ge 3$ and $1\le a\le m$ with $\gcd(a,m)=1$, then $m$ cannot be a Carmichael-Wieferich number with bases both $a$ and $m-a$.
\end{proposition}
\begin{proof}
Notice that $\lambda(m)$ is even when $m\ge 3$. By Proposition \ref{fun} (2), we have
$$
C_{m}(m-a)\equiv C_{m}(a)-\frac{\lambda(m)}{a}\quad ({\rm mod}\, m).
$$
Then, the desired result comes from $\lambda(m)<m$.  
\end{proof}

\begin{corollary}
If $m\ge 3$, define the set $S_{m}=\{a: 1\le a\le m, \gcd(a,m)=1, m$ is a Carmichael-Wieferich number with base $a\}$.
Then $|S_{m}|\le \varphi(m)/2$.
\end{corollary}

By Proposition \ref{fun} (2), for any $\gcd(b,m)=1$, there exists $1\le a\le m^{2}$ with
$b\equiv a $ (mod $m^{2}$), such that
$$
C_{m}(b)\equiv C_{m}(a)\quad ({\rm mod}\, m).
$$
Hence, if we want to determine with which base $m$ can be a Carmichael-Wieferich number, we only need to consider $1\le a\le m^{2}$.

Assume that $m$ has  the prime factorization $m=p_{1}^{r_{1}}\cdots p_{k}^{r_{k}}$.  
In \cite[Proposition 4.4]{ADS} the authors have used the Euler quotient $Q_m$ to define a homomorphism from $(\mathbb{Z}/m^{2}\mathbb{Z})^{*}$ 
to $(\mathbb{Z}/m\mathbb{Z},+)$, whose image is $d\mathbb{Z}/m\mathbb{Z}$, where 
\begin{equation} \label{d(m)}
d=\prod_{i=1}^{k}d_{i} 
\quad \textrm{and} \quad 
d_{i}=\left\{ \begin{array}{ll}
                 \gcd(p_{i}^{r_{i}}, 2\varphi(m)/\varphi(p_i^{r_i})) & \textrm{if $p_{i}=2$ and $r_{i}\ge 2$},\\

                 \gcd(p_{i}^{r_{i}}, \varphi(m)/\varphi(p_i^{r_i})) & \textrm{otherwise}.
                 \end{array} \right.
\end{equation}
Here, we can do similar things using the Carmichael quotient and applying the same strategy as in \cite{ADS}. 

By Proposition \ref{fun}, the Carmichael quotient $C_{m}(x)$ induces a homomorphism
$$
\phi_m:(\mathbb{Z}/m^{2}\mathbb{Z})^{*}\to (\mathbb{Z}/m\mathbb{Z},+), x\mapsto C_m(x).
$$

\begin{proposition}\label{hom}
Let $m=p_{1}^{r_{1}}\cdots p_{k}^{r_{k}}$ be the prime factorization of $m \ge 2$. For $1\le i\le k$, put
\begin{equation}
d^{\prime}_{i}=\left\{ \begin{array}{ll}
                 \gcd(p_{i}^{r_{i}}, 2\lambda(m)/\lambda(p_i^{r_i})) & \textrm{if $p_{i}=2$ and $r_{i}= 2$},\\

                 \gcd(p_{i}^{r_{i}}, \lambda(m)/\lambda(p_i^{r_i})) & \textrm{otherwise}.
                 \end{array} \right.
\notag
\end{equation}
Let $d^{\prime}=\prod_{i=1}^{k}d^{\prime}_{i}$. Then the image of the homomorphism $\phi_m$ is $d^{\prime}\mathbb{Z}/m\mathbb{Z}$. 
\end{proposition}

\begin{proof}
We show the desired result case by case. 

(I) First we prove the result for the case $k=1$, that is $m=p^r$, where $p$ is a prime and $r$ is a positive integer.  

Suppose that $p=2$. If $r=2$, then $C_m(3)=2$, and for any positive integer $n$ we have $C_m(2n+1)= n(n+1)$, which is even, so the image of $\phi_m$ is $2\Z/m\Z$. 
On the other hand, if $r = 1$ or $r \ge 3$, since $C_2(3)=1$ and $C_8(3)=1$, by using Proposition \ref{reduction} we see that $C_m(3)$ is an odd integer, 
so the image of $\phi_m$ is $\Z/m\Z$. 

Now, assume that $p>2$. Note that $C_p(p+1) \equiv -1  \pmod p$, by Proposition \ref{reduction} we have  $C_m(p+1) \equiv -1  \pmod p$, 
which implies that $p \nmid C_m(p+1)$. Thus, there exists a positive integer $n$ such that $n C_m(p+1) \equiv 1 \pmod m$. 
Then, by Proposition \ref{fun} (1) we deduce that  $C_m((p+1)^n) \equiv 1 \pmod m$. 
So, the image of $\phi_m$ is $\Z/m\Z$. 
 
(II) To complete the proof, we prove the result when $k \ge 2$. 

For simplicity, denote $m_i = m/p_i^{r_i}$ and $n_i = \lambda(m) / \lambda(p_i^{r_i})$ for each $1\le i \le k$, 
and then let $m_i^{\prime}$  be an integer such that $m_i^{2}m_i^{\prime} \equiv 1 \pmod{p_i^{r_i}}$. 
By Proposition \ref{factor}, we have 
\begin{equation}\label{Cma}
C_m(a) \equiv \sum_{i=1}^{k} m_i m_i^{\prime} n_i C_{p_i^{r_i}}(a) \pmod{m}. 
\end{equation}
So, for each $1\le i \le k$, $C_m(a) \equiv  m_i m_i^{\prime} n_i C_{p_i^{r_i}}(a) \pmod{p_i^{r_i}}$. 
If $p_i=2$ and $r_i=2$, note that for any odd integer $a>1$, $C_4(a)$ is even, then we see that $d_i^{\prime} \mid n_i C_{p_i^{r_i}}(a)$, 
and thus $d_i^{\prime} \mid C_m(a)$. 
Otherwise if $p_i>2$ or $r_i \ne 2$, then $d_i^{\prime} \mid n_i$, and so $d_i^{\prime} \mid C_m(a)$. 
Hence, we have $d^{\prime} \mid C_m(a)$ for any integer $a$ coprime to $m$. 

Let $b = \gcd(m, m_1m_1^{\prime}n_1, \ldots, m_km_k^{\prime}n_k)$. 
Then, there exist integers $X_1,\ldots, X_k$ such that 
\begin{equation} \label{bm}
b \equiv \sum_{i=1}^{k} m_i m_i^{\prime} n_i X_i \pmod{m}. 
\end{equation}
If we denote $b_i = \gcd(p_i^{r_i}, m_i m_i^{\prime}n_i)$ for each $1 \le i \le k$, then $b = \prod_{i=1}^{k} b_i$, 
here we remark that $b_i = \gcd(p_i^{r_i}, n_i)$. 
It is easy to see that for each $1 \le i \le k$, if $p_i>2$ or $r_i \ne 2$, we have $d_i^{\prime} = b_i$. 
Further, when $p_i=2$ and $r_i=2$, $d_i^{\prime} = 2b_i$ if $8 \nmid \lambda(2p_1\ldots p_k)$, and $d_i^{\prime} = b_i$ otherwise.  

We now have three cases for $m$: 
\begin{enumerate}
\item[(i)] There exists $1\le j \le k$ such that $p_j=2, r_j=2$ and 
$$
8 \nmid \lambda(2p_1\ldots p_k). 
$$
\item[(ii)] There exists $1\le j \le k$ such that $p_j=2, r_j=2$ and 
$$
8 \mid \lambda(2p_1\ldots p_k). 
$$

\item[(iii)] All the other cases. 
\end{enumerate}
Clearly, in Cases (ii) and (iii) we have $d^{\prime}=b$, and in Case (i) $d^{\prime} = 2b$.

According to (I),  there exist integers $a_i$ with $p_i \nmid a_i$ for $1 \le i \le k$ defined by 
\begin{equation*}
C_{p_i^{r_i}}(a_{i}) \equiv \left\{ \begin{array}{ll}
                 2X_i & \textrm{in Case (i)}, \\
                 X_i & \textrm{in Case (iii)}, \\
                 & \qquad\qquad\qquad\qquad\qquad  \pmod{p_i^{r_i}} \\
                 X_i & \textrm{in Case (ii) and $i\ne j$}, \\
                 0 & \textrm{in Case (ii) and $i=j$}.     
                 \end{array} \right.
\end{equation*}

By the Chinese Remainder Theorem, we can choose a positive integer $a$ such that $a \equiv a_i \pmod{p_i^{2r_i}}$. 
So, by Proposition \ref{fun} (2) we have $C_{p_i^{r_i}}(a) \equiv C_{p_i^{r_i}}(a_i) \pmod{p_i^{r_i}}$. 
Then, combining with \eqref{bm} and the relation between $b$ and $d^{\prime}$, we obtain 
$m_im_i^{\prime}n_i C_{p_i^{r_i}}(a) \equiv d^{\prime} \pmod{p_i^{r_i}}$ for each $1\le i \le k$ in all the three cases.  
Finally, using \eqref{Cma} we have $C_m(a) \equiv d^{\prime} \pmod{m}$, which completes the proof. 
\end{proof}

Comparing \eqref{d(m)} with Proposition \ref{hom}, we have $d^{\prime} \mid d$. 
Moreover, by Proposition \ref{relation} we get 
$$
\frac{\varphi(m)}{\lambda(m)}d^{\prime}\Z/m\Z = d\Z/m\Z, 
$$
which implies that $\gcd(\frac{\varphi(m)}{\lambda(m)}d^{\prime},m)=d$. 

In Proposition \ref{hom}, if choosing $m=2^r$ with $r\ge 3$, we have $d=2$ and $d^{\prime}=1$; while choosing $m=2^{r_1}p^{r_2}$ with $r_1\ge 3$ and odd prime $p\equiv 3$ (mod 4), we have $d=4$ and $d^{\prime}=1$. Hence, compared with \cite[Proposition 4.4]{ADS}, the homomorphism $\phi_m$ can be surjective in  more cases. 

For any integer $m\ge 2$, we define the set 
\begin{align*}
T_{m}=\{a: &1\le a\le m^{2}, \gcd(a,m)=1, \\ 
& \textrm{$m$ is a Carmichael-Wieferich number with base $a$}\}.
\end{align*}
 Actually, $T_m$ is the kernel of the homomorphism $\phi_m$, then the following result follows directly from Proposition \ref{hom}.

\begin{corollary}\label{exist}
We have $|T_{m}|=d^{\prime}\varphi(m)$, where $d^{\prime}$ is defined in Proposition \ref{hom}.
\end{corollary}

Corollary \ref{exist} shows that any integer $m\ge 2$ can be a Carmichael-Wieferich number with some base. However, the next proposition suggests that such Carmichael-Wieferich numbers are rare.

\begin{proposition} \label{limit}
We have $\lim\limits_{m\to \infty}\frac{|T_{m}|}{\varphi(m^{2})}=0.$
\end{proposition}
\begin{proof}
Denote by $d(m)$ the parameter $d$ in \eqref{d(m)}. By Corollary \ref{exist}, we know that 
$$
\frac{|T_{m}|}{\varphi(m^{2})}\le\frac{d(m)}{m}.
$$
So, it suffices to prove that $\lim\limits_{m\to \infty}\frac{d(m)}{m}=0$.

For primes $p$, we have
$$
\lim\limits_{p\to \infty}\frac{d(p)}{p}=\lim\limits_{p\to \infty}\frac{1}{p}=0.
$$
So $\liminf\limits_{m\to \infty}\frac{d(m)}{m}=0.$

Suppose that $\limsup\limits_{m\to \infty}\frac{d(m)}{m}\ne0.$ Then there exists a subsequence $\{\frac{d(n_{i})}{n_{i}}\}$
such that $\lim\limits_{i\to \infty}\frac{d(n_{i})}{n_{i}}=\limsup\limits_{m\to \infty}\frac{d(m)}{m}\ne0.$

For an integer $m\ge 2$, let $m=p_{1}^{r_{1}}\cdots p_{k}^{r_{k}}$ be its prime factorization.
Put $\alpha_{m}=\max\{r_{1},\cdots,r_{k}\}$. Here we use the
notation in \eqref{d(m)}. For each $1\le j\le k$, we have $\frac{d(m)}{m}\le d_{j}/p_{j}^{r_{j}}$.
In particular, if $p_{j}$ is the largest prime factor of $m$, then $\frac{d(m)}{m}\le 2/p_{j}^{r_{j}}$.

For each $i$, let $p_{i}$ be the largest prime factor of $n_{i}$, we abbreviate $\alpha_{n_{i}}$ to $\alpha_{i}$.
Since $\frac{d(n_{i})}{n_{i}}\le \frac{2}{p_{i}}$ for each $i$ and $\lim\limits_{i\to \infty}\frac{d(n_{i})}{n_{i}}\ne 0$, there must exist an integer $q$ such that
$p_{i}<q$ for all $i$. Put $\beta=2\prod\limits_{\substack{2\le p<q\\ \textrm{$p$ prime}}}(p-1)$.
Since $d(n_{i})\le \beta$, we have $\frac{d(n_{i})}{n_{i}}\le \frac{\beta}{2^{\alpha_{i}}}$ for each $i$.
Notice that $n_{i}\to \infty$ when $i\to \infty$, we must have $\alpha_{i}\to \infty$ as $i\to \infty$.
Hence, we have $\lim\limits_{i\to \infty}\frac{d(n_{i})}{n_{i}}=0$. This leads to a contradiction.

So, we have $\limsup\limits_{m\to \infty}\frac{d(m)}{m}=0$. This completes the proof.   
\end{proof}

Assume that there are infinitely many Sophie Germain primes. We construct a sequence $\{n_{i}\}$ with $n_i=p_i(2p_i+1)$, where $p_i$ is a Sophie Germain prime, and then $2p_i+1$ is also a prime. It is easy to see that $d(n_i)=p_i$ and $\lim\limits_{i\to \infty}\frac{d(n_{i})}{\sqrt{n_{i}}}=\frac{1}{\sqrt{2}}$. This implies that the limit $\lim\limits_{m\to \infty}\frac{d(m)}{\sqrt{m}}=0$ may be not true in general.

In the sequel, we want to characterize all the Carmichael-Wieferich numbers.

Let $p$ be a prime and $a$ an integer with $p\nmid a$. Put
$$
\sigma(a,p)=\ord_{p}(a^{p-1}-1)-1\quad \textrm{if $p$ is odd};
$$
\begin{equation}
\sigma(a,2)=\left\{ \begin{array}{ll}
                 \ord_{2}(a-1)-1 & \textrm{\quad if $a\equiv1$ (mod 4)},\\

                 \ord_{2}(a+1)-1 & \textrm{\quad if $a\equiv3$ (mod 4)}.
                 \end{array} \right.
\notag
\end{equation}
Then, we can state an analogue of \cite[Proposition 5.4]{ADS}. For the convenience of the reader, we reproduce the proof. 

\begin{proposition} \label{order}
Let $\gcd(a,m)=1$, and $m=p_{1}^{r_{1}}\cdots p_{k}^{r_{k}}$ be the prime factorization of $m\ge 3$.
Fix an integer $j$ with $1\le j\le k$, let $p=p_{j}$ and $r=r_{j}$. If $p\ne 2$ or $r\le 2$, put
\begin{equation}
n=\left\{ \begin{array}{ll}
                 0 & \textrm{if $\ord_{p}\lcm\left(p_{1}-1,\cdots,p_{k}-1\right)\le r-1$},\\

                 \ord_{p}\lcm\left(p_{1}-1,\cdots,p_{k}-1\right)-r+1 & \textrm{otherwise};
                 \end{array} \right.
\notag
\end{equation}
otherwise if $p=2$ and $r> 2$, put
\begin{equation}
n=\left\{ \begin{array}{ll}
                 0 & \textrm{if $\ord_{p}\lcm\left(p_{1}-1,\cdots,p_{k}-1\right)\le r-2$},\\

                 \ord_{p}\lcm\left(p_{1}-1,\cdots,p_{k}-1\right)-r+2 & \textrm{otherwise}.
                 \end{array} \right.
\notag
\end{equation}
Moreover, put
\begin{equation}
e(m,p)=\left\{ \begin{array}{ll}
                 n & \textrm{if $p\ne 2$ or $r\le 2$},\\

                 n-1 & \textrm{otherwise}.
                 \end{array} \right.
\notag
\end{equation}
Then we have
$$
\ord_{p}C_{m}(a)=e(m,p)+\sigma(a,p).
$$
\end{proposition}
\begin{proof}
Notice that  $\lambda(m)=p^{n}\lambda(p^{r})X$, where $X$ is an integer with $p\nmid X$. Put $b=a^{p^{n}\lambda(p^{r})}$. 
Then, since 
$$
a^{\lambda(m)}-1=b^X-1=(b-1)\sum_{i=0}^{X-1}b^i,
$$
$b\equiv 1$ (mod $p$) and $\sum_{i=0}^{X-1}b^i \equiv X \not \equiv 0$ (mod $p$), we obtain
$$
\ord_p(a^{\lambda(m)}-1) =\ord_p(b-1)=\ord_p(a^{p^{n}\lambda(p^{r})}-1). 
$$
Thus, if $p$ is an odd prime, by using \cite[Lemma 5.1]{ADS} we have 
$$
\ord_p(a^{\lambda(m)}-1)=\ord_p((a^{p-1})^{p^{n+r-1}}-1)
=\ord_p(a^{p-1}-1)+n+r-1,
$$
which implies that 
$$
\ord_{p}C_{m}(a)=e(m,p)+\sigma(a,p).
$$
Similarly, applying \cite[Lemmas 5.1 and 5.3]{ADS}, one can verify the remaining case $p=2$ by noticing that $m\ge 3$.  
\end{proof}

The next proposition, a criterion for a number $m$ being a Carmichael-Wieferich number, follows directly from Proposition \ref{order}. 
\begin{proposition}\label{cri}
Let $\gcd(a,m)=1$, and $m=p_{1}^{r_{1}}\cdots p_{k}^{r_{k}}$ be the prime factorization of $m\ge 3$.
Then the following statements are equivalent:

{\rm (1)} $m$ is a Carmichael-Wieferich number with base $a$,

{\rm (2)} $e(m,p_{j})+\sigma(a,p_{j})\ge r_{j}$, for any $1\le j\le k$.
\end{proposition}

Although it is known that Wieferich primes exist for many different bases (see \cite{Mon}), the following problem is
still open.
\begin{center}
\it{Whether Wieferich primes exist for all bases?}
\end{center}

\begin{proposition}
For a non-zero integer $a$, if there exists a Carmichael-Wieferich number $m$ with base $a$ and $m$ has
an odd prime factor, then there exists a Wieferich prime with base $a$.
\end{proposition}
\begin{proof}
Let $m=p_{1}^{r_{1}}\cdots p_{k}^{r_{k}}$ be the prime factorization of $m$ with $p_{1}<p_{2}<\cdots<p_{k}$, where $p_k$ is an odd prime.
Since $e(m,p_k)=0$ and $m$ is a Carmichael-Wieferich number with base $a$, by Proposition \ref{cri} we have $\sigma(a,p_{k})\ge r_{k}\ge 1$.
Notice that $p_{k}$ is an odd prime, so $p_{k}$ is a Wieferich prime with base $a$.  
\end{proof}

Finally, we want to remark that a Carmichael-Wieferich number $m$ with base $a$ is also a Wieferich number with base $a$, but the converse is not true.

\begin{example}
{\rm From Table 1 of \cite{Mon}, 3 and 7 are two Wieferich primes with base 19. It is straightforward
to see that 2 is not a Wieferich prime with base 19. By \cite[Theorem 5.5]{ADS},
$m=2^{2}\cdot 3\cdot 7$ is a Wieferich number with base 19. But by Proposition \ref{cri}, $m$ is not
a Carmichael-Wieferich number with base 19.}
\end{example}

\section{Involving perfect nonlinear function}
Let $(A,+)$ and $(B,+)$ be two additive abelian groups, and denote by $\bar{A}$ the set of non-identity elements of $A$. When $|A|$ is a multiple of $|B|$, we can consider the following definition; see \cite{Ding2004} for more details.
\begin{definition}
{\rm
Let $f: A\to B$ be a function from $A$ to $B$. Then $f$ is called \emph{perfect nonlinear} if for every $(a,b)\in \bar{A}\times B$,
$|\{x\in A: f(x+a)-f(x)=b\}|=\frac{|A|}{|B|}$.
}
\end{definition}

Perfect nonlinear functions have important applications in cryptography, sequences and coding theory. For example, as in \cite{Ding2003}, such functions can be used to construct authentication codes.

For the homomorphism
$\phi_m:(\mathbb{Z}/m^{2}\mathbb{Z})^{*}\to (\mathbb{Z}/m\mathbb{Z},+)$, defined in Section \ref{number}, we extend its definition to those integers $a$ with $\gcd(a,m)\ne 1$ by defining $\phi_m(a)=0$. Then we get a function 
 $$ 
 f_m : (\mathbb{Z}/m^{2}\mathbb{Z},+) \to (\mathbb{Z}/m\mathbb{Z},+), x \mapsto \phi_m(x).
 $$ 
 For this function $f_m$, we have the following proposition.
\begin{proposition}
The function $f_m$ is perfect nonlinear if and only if $m$ is a prime number.
\end{proposition}
\begin{proof}
First, suppose that $m$ is a prime number. By \cite[Lemma 8]{Ding2003} (or \cite[Theorem 48]{Ding2004}) and Proposition \ref{hom}, it is easy to show that $f_m$ is perfect nonlinear.

Now assume that $m$ is a composite integer. Let $p$ be a prime factor of $m$. Notice that $f_m(kp)=0$ for any $k\ge 1$, and $(m+2)p\le m(m+2)/2 < m^2$. Then choosing $(p,0)\in \mathbb{Z}/m^{2}\mathbb{Z} \times \mathbb{Z}/m\mathbb{Z}$, we obtain 
\begin{align*}
|\{x\in \mathbb{Z}/m^{2}\mathbb{Z}: f_m(x+p)-f_m(x)=0\}|
&\ge |\{x=kp: 1\le k \le m+2\}|\\
&= m+2>m.
\end{align*}
 By definition, the function $f_m$ is not perfect nonlinear.  
\end{proof}

Thus, the function $f_m$ gives a new kind of perfect nonlinear functions when $m$ is a prime number. Furthermore, this kind of perfect nonlinear functions is much more convenient for computations than that given in \cite[Example 49]{Ding2004}.

\section*{acknowledgements}
The author would like to thank Professor Arne Winterhof for sending him  the recent work \cite{CW}. 
He wants to thank the referee for careful reading and valuable comments. 
He is also grateful to the referee of his recent joint paper \cite{LSS} for pointing our the error in the previous Proposition \ref{hom}.

\end{document}